\documentclass{article}

\usepackage[T1]{fontenc}
\usepackage[letterpaper,margin=1in]{geometry} 
\usepackage{mathtools, verbatim, 
  rotating, comment, epic, amsmath, amscd, amsfonts, amsthm, amssymb, enumerate, graphicx, cases, tablefootnote}
\usepackage{url}
\newtheorem{thm}{Theorem}
\newtheorem{lem}[thm]{Lemma}
\newtheorem{conj}[thm]{Conjecture}
\theoremstyle{remark}

\newcommand{\CC}{\mathbb{C}}
\newcommand{\FF}{\mathbb{F}}
\newcommand{\QQ}{\mathbb{Q}}
\newcommand{\PP}{\mathbb{P}}
\newcommand{\RR}{\mathbb{R}}
\newcommand{\ZZ}{\mathbb{Z}}
\newcommand{\C}{\mathcal{C}}
\renewcommand{\L}{\mathcal{L}}
\newcommand{\M}{\mathcal{M}}
\renewcommand{\S}{\mathcal{S}}
\newcommand{\WW}{\mathcal{W}}
\newcommand{\X}{\mathcal{X}}

\newcommand{\cross}{\times}
\newcommand{\size}[1]{\lvert #1 \rvert}
\newcommand{\isom}{\cong}
\renewcommand{\(}{\left(}
\renewcommand{\)}{\right)}
\renewcommand{\epsilon}{\varepsilon}

\numberwithin{equation}{section}

\title{Large orbits on Markoff-type K3 surfaces over finite fields}
\author{Evan M. O'Dorney}

\begin{document}
  
  \maketitle

\begin{abstract}%
We study the surface $\mathcal{W}_k : x^2 + y^2 + z^2 + x^2 y^2 z^2 = k x y z$ in $(\mathbb{P}^1)^3$, a tri-involutive K3 (TIK3) surface. We explain a phenomenon noticed by Fuchs, Litman, Silverman, and Tran: over a finite field of order $\equiv 1$ mod $8$, the points of $\mathcal{W}_4$ do not form a single large orbit under the group $\Gamma$ generated by the three involutions fixing two variables and a few other obvious symmetries, but rather admit a partition into two $\Gamma$-invariant subsets of roughly equal size. The phenomenon is traced to an explicit double cover of the surface.
\end{abstract}


\section{Introduction}

A surface $\WW$ defined by a general $(2, 2, 2)$-form on $\PP^1 \cross \PP^1 \cross \PP^1$ is a K3 surface, possibly singular, admitting three involutions $\sigma_x, \sigma_y, \sigma_z$ that fix two coordinates and flip the third to the other root of the appropriate quadratic equation. In \cite{SilvermanK3}, Fuchs, Litman, Silverman and Tran call these \emph{tri-involutive K3 (TIK3)} surfaces and ask the natural question of the orbit structure of the $\FF_q$-points of such a $\WW$ under the group $\C_2 \ast \C_2 \ast \C_2$ generated by the $\sigma_i$. For such a large group, in the absence of additional structure, one normally conjectures that for $q$ large, the points form a single large orbit and possibly some very small orbits.

A TIK3 surface is given by $27$ coefficients. It is reasonable to begin with families that admit extra symmetries. For instance, the \emph{Markoff surface}
\[
  \M : x^2 + y^2 + z^2 = 3 x y z,
\]
whose $\ZZ$-points arise in Diophantine approximation, has an action by $\S_3 \ltimes (\C_2 \cross \C_2) \isom \S_4$ by permuting the coordinates and negating an even number of them. In 1879, Markoff \cite{MarkoffI} showed that the $\ZZ$-points $\M(\ZZ)$ form a single orbit under the group $\Gamma$ generated by all the symmetries named. In 1991, Baragar \cite{Baragar} conjectured that, after removing the singular point $(0,0,0)$, $\M(\FF_p)$ consists of one $\Gamma$-orbit for all primes $p$. In 2016, Bourgain, Gamburd, and Sarnak \cite{BGS16_short, BGS_I} have proved that all orbits but one have size $O_\epsilon(q^\epsilon)$ for all $\epsilon > 0$. Recently, Chen \cite{ChenMarkoff} announced a proof that Baragar's conjecture holds (that is, there is only one nonsingular $\Gamma$-orbit) for all but finitely many $p$.

An even richer symmetry group occurs for the family
\[
  \WW_k : x^2 + y^2 + z^2 + x^2 y^2 z^2 = k x y z,
\]
in which one can also reciprocate an even number of the coordinates. Via numerical experiments, the authors of \cite{SilvermanK3} found that the case $k = \pm 4$, $q \equiv 1$ mod $8$ is exceptional in that there are \emph{two} largest orbits of roughly equal size. They conjecture that there is an algebro-geometric reason for this aberration and that it is the only one in its family:

\begin{conj}[\cite{SilvermanK3}, \textsection 13]
  Let $q$ be an odd prime power. Let $\Gamma$ be the subgroup of automorphisms of $\WW_k$ generated by the involution $\sigma_x$, the finite-order symmetries $(x,y,z) \mapsto (x,-y,-z)$ and $(x,y,z) \mapsto (x,y^{-1}, z^{-1})$, and the permutations of the coordinates. Then the large $\Gamma$-orbits of the $\FF_q$-points of $\WW_k$ are as follows:
  \begin{enumerate}[$($a$)$]
    \item For $k = \pm 4$ or $\pm 4\sqrt{-1}$, if $q \equiv 1$ mod $8$, there are two orbits of size $\frac{1}{2} q^2 + o(q^2)$.
    \item In all other cases, there is one orbit of size $q^2 + o(q^2)$.
  \end{enumerate}
  All other orbits are necessarily of size $o(q^2)$.
\end{conj}

\section{Main theorem}

In this paper, we focus our attention on $\WW_4$. (The surfaces $\WW_k$ and $\WW_{-k}$ are isomorphic via negating the coordinates, and if the ground field contains a square root $i$ of $-1$, then $\WW_k$ and $\WW_{ik}$ are also isomorphic via multiplying all coordinates by $-i$.) Incidentally, over $\RR$, the variety
\[
  \WW_4 : x^2 + y^2 + z^2 + x^2 y^2 z^2 = 4 x y z
\]
is an equality case of the four-term AM--GM inequality
\[
  \frac{x^2 + y^2 + z^2 + x^2 y^2 z^2}{4} \geq \sqrt[4]{x^2 \cdot y^2 \cdot z^2 \cdot x^2 y^2 z^2} = \size{x y z}.
\]
In particular, $\WW_4$ has no real points other than the eight singular points, which are $(0,0,0)$, $(0,\infty,\infty)$, $(1,1,1)$, $(1,-1,-1)$, and their images under permutations of coordinates. There is also an extra symmetry of $\WW_4$ given by the linear fractional transformation $t \mapsto (1 + t)/(1 - t)$ on each coordinate. This symmetry will not be needed in this note;  adjoining it to $\Gamma$ does not change our results.

\begin{thm}
Let $\Gamma$ be the subgroup of the automorphisms of $\WW_4$ generated by those mentioned above:
\begin{itemize}
  \item The transpositions $\tau_{yz}, \tau_{xz}, \tau_{xy}$ of two coordinates;
  \item The nontrivial automorphism $\sigma_x$ fixing $y$ and $z$;
  \item The sign change $s_x$ that negates $y$ and $z$;
  \item The transformation $r_x$ that reciprocates $y$ and $z$.
\end{itemize}
(Observe that conjugates like $\sigma_y = \tau_{xy} \circ \sigma_x \circ \tau_{xy}^{-1}$ also belong to $\Gamma$.)
If $q \equiv 1$ mod $8$, then there is a partition $\WW_4(\FF_q) = U_1 \sqcup U_2$ into two disjoint $\Gamma$-invariant subsets, each of roughly the same size $\frac{1}{2} q^2 + O(q^{3/2})$. In particular, every $\Gamma$-orbit has size at most $\frac{1}{2} q^2 + O(q^{3/2})$.
\end{thm}

\begin{proof}
To prove this theorem, we produce a $\Gamma$-invariant function on the $\FF_q$-points of $\WW_4$. As the condition $q \equiv 1$ mod $8$ suggests, this invariant will be defined over the $8$th cyclotomic field $\QQ(\zeta_8) = \QQ(i, \sqrt{2})$.

Let $\Delta_x$ be the discriminant of the equation
\[
  f = f(x,y,z) = x^2 + y^2 + z^2 + x^2 y^2 z^2 - 4 x y z
\]
of $\WW_4$ considered as a quadratic in $x$. Note that $\Delta_x$ is reducible over $\ZZ[i]$:
\begin{align*}
  \Delta_x &= (4yz)^2 - 4(1 + y^2 z^2)(y^2 + z^2) \\
  &= 4\(y(1 - z^2) + i z (1 - y^2)\) \cdot \(y(1 - z^2) - i z (1 - y^2)\).
\end{align*}
We denote by $\Delta_{yz}$ the first factor
\[
  \Delta_{yz} = y(1 - z^2) + i z (1 - y^2)
\]
of the discriminant. We claim that it is our desired invariant:
\begin{lem}
  Let $F$ be a field, $2 \neq 0 \in F$, containing all $8$th roots of unity. The class $\Delta$ of $\Delta_{yz}$, in the multiplicative group
  \[
    F(\WW_4)^\cross / \(F(\WW_4)^\cross\)^2
  \]
  of the function field $F(\WW_4)$ modulo squares, is invariant under the group $\Gamma$.
\end{lem}

\begin{proof}
Invariance under $\sigma_x$ is clear since $\sigma_x$ preserves the $y$ and $z$ coordinates. Invariance under $s_x$ and $r_x$ is clear using $i \in F$:
\[
  s_x^*(\Delta_{yz}) = -\Delta_{yz} = i^2 \cdot \Delta_{yz}, \quad
  r_x^*(\Delta_{yz}) = \(\frac{i}{yz}\)^2 \Delta_{yz}.
\]
It remains to show that $\Delta$ is invariant under the transpositions $\tau_{uv}$ of coordinates. Denote by $\Delta_{zy}$, $\Delta_{xz}$, etc{.} the corresponding expressions where the indicated variables are substituted in place of $y$ and $z$ in $\Delta_{yz}$.

First consider $\tau_{yz}$, that is, compare $\Delta_{yz}$ and $\Delta_{zy}$. Their product is
\begin{align*}
  \Delta_{yz} \cdot \Delta_{zy}
  &= \(y(1 - z^2) + i z (1 - y^2)\) \cdot \(z(1 - y^2) + i y (1 - z^2)\) \\
  &= \(y(1 - z^2) + i z (1 - y^2)\) \cdot i\(y(1 - z^2) - i z (1 - y^2)\) \\
  &= \frac{i}{4} \Delta_x.
\end{align*}
Now $\Delta_x = \(x - \sigma_x^*(x)\)^2$ is a square in the function field. Also $i = \zeta_8^2$ is a square, so $\Delta_{yz}$ and $\Delta_{zy}$ define the same class in $F(\WW_4)^\cross / \(F(\WW_4)^\cross\)^2$.

Next look at $\tau_{xy}$, in other words compare $\Delta_{yz}$ and $\Delta_{xz}$. The product $\Delta_{yz} \cdot \Delta_{xz}$ has degree $(2, 2, 4)$ as a function of each variable separately, while $f$ has degree $(2, 2, 2)$, so it is natural to seek a quadratic $q(z) = a z^2 + b z + c$ in $z$ alone such that we have the identity
\[
  \Delta_{xz} \Delta_{yz} + q(z) f(x, y, z) = g(x, y, z)^2
\]
for some polynomial $g$ of degree $(1, 1, 2)$. Comparing the terms not involving $y$, we quickly find that $q(z) = \pm \frac{1}{2}\(z^2 \pm 2z - 1\)$ (the $\pm$ signs are independent). A short computer computation shows that indeed
\[
  -2 \Delta_{xz} \Delta_{yz} + (z^2 + 2z - 1) f = 
   (x y z^2 + x y z - i x z - i y z - z^2 + i x + i y - z)^2
\]
as an identity in $\ZZ[i, x, y, z]$. Since $\sqrt{-2} = \zeta_8 + \zeta_8^3 \in F$, this shows the desired invariance under $\tau_{xy}$, completing the proof that $\Delta$ is $\Gamma$-invariant.
\end{proof}

It is necessary to check that $\Delta$ is not constant, that is, $\Delta_{yz}$ is not a constant times a square in $F(\WW_4)$. To this end, note that $\WW_4$ contains the line $\L = \{\(iy, y, 0\)\}$. We have $\Delta_{yz} |_{\L} = y$, which is not a square in $F(\L)$.

We can use $\Delta$ to define a double cover of $\WW_4$. For a concrete construction, let $\X_0$ be the subvariety of $\(\PP^1\)^9$ defined by (the projectivization of) the following affine equations:
\[
  \X_0 = \left\{(x,y,z, \delta_{xy}, \delta_{xz}, \delta_{yx}, \delta_{yz}, \delta_{zx}, \delta_{zy}) : 
  f(x,y,z) = 0 \text{ and } \delta_{uv}^2 = \Delta_{uv} \text{ for all distinct } u, v \in \{x,y,z\} \right\}.
\]
Then let $\X$ be an irreducible component of $\X_0$. Since the ratios of the $\Delta_{uv}$ are squares but the $\Delta_{uv}$ themselves are not, we see that the natural projection $\pi \colon \X \to \WW_4$ is a double cover. We claim that $\pi$ is ramified only at the eight singular points. When $x, y, z$ are finite, ramification can only occur when all $\Delta_{uv}$ are zero. This happens only when $\Delta_x = \Delta_y = \Delta_z = 0$, that is, $(x,y,z)$ is a singular point. If $x = \infty$, then $y = \pm i/z$, so $\Delta_{yz}$ is nonzero unless $y$ or $z$ is infinite, which occurs only at the singular points.

Let $U$ be the image of
\[
  \pi : \X(\FF_q) \to \WW_4(\FF_q);
\]
in other words, $U$ is the set of $(x,y,z) \in \WW_4(\FF_q)$ such that some $\Delta_{uv}$ is a nonzero square, plus the singular points. Note that $U$ is $\Gamma$-invariant, since $\Delta$ is $\Gamma$-invariant. Moreover, $\pi$ is exactly $2$-to-$1$ onto $U$, except at the eight singular points, which each have only one preimage. Thus, using the Weil conjectures for $\X$, we have
\[
  \size{U} = \frac{\size{\X(\FF_q)} + 8}{2} = \frac{1}{2} q^2 + O\(q^{3/2}\),
\]
yielding the desired splitting
\[
  \WW_4(\FF_q) = U \sqcup \big(\WW_4(\FF_q) \setminus U\big).
\]
The implied constant is absolute and depends only on the topology of $\X$ considered as a variety over $\CC$, which we do not study here.
\end{proof}

For $q \equiv 5$ mod $8$, the same proof yields:
\begin{thm}
Let $q \equiv 5$ mod $8$ be a prime power. Let $\Gamma' \subset \Gamma$ be the subgroup of index $2$ generated by $\sigma_x$, $s_x$, $r_x$, and the $3$-cycle $\tau_{xy}\tau_{yz}$. Then there is a partition of the nonsingular points of $\WW_4(\FF_q)$ into two disjoint $\Gamma'$-invariant subsets of the same size, which are interchanged by $\tau_{xy}$. In particular, every $\Gamma'$-orbit on $\WW_4(\FF_q)$ has size at most $\frac{1}{2} q^2 + O\(q^{3/2}\)$.
\end{thm}

\bibliography{../Master.bib}
\bibliographystyle{plain}

\end{document}